\documentclass[10pt,reqno]{amsart}

\usepackage{amsmath,amssymb,amsfonts}

\usepackage{	amsthm}
\newtheorem{theorem}{Theorem}
\newtheorem{proposition}{Proposition}
\newtheorem{lemma}{Lemma}
\newtheorem{corollary}{Corollary}
\newtheorem{definition}{Definition}
\newtheorem{remark}{Remark}

\usepackage{mathrsfs}
\usepackage{enumerate}

\usepackage[bitstream-charter]{mathdesign}
\usepackage[T1]{fontenc}

\newcommand\bB{\mathbb{B}}
\newcommand\bE{\mathbb{E}}
\newcommand\bH{\mathbb{H}}
\newcommand\bL{\mathbb{L}}
\newcommand\bN{\mathbb{N}}
\newcommand\bP{\mathbb{P}}
\newcommand\bX{\mathbb{X}}
\newcommand\bZ{\mathbb{Z}}

\newcommand\fR{\mathbf{R}}
\newcommand\R{\mathbf{R}}
\newcommand\fN{\mathbf{N}}

\newcommand\cF{\mathcal{F}}
\newcommand\cP{\mathcal{P}}
\newcommand\cR{\mathcal{R}}
\newcommand\cS{\mathcal{S}}
\newcommand\cM{\mathcal{M}}
\newcommand\cO{\mathcal{O}}

\newcommand\rF{\mathscr{F}}

\newcommand{\bb}{\mathbf{b}}
\newcommand{\bu}{\mathbf{u}}
\newcommand{\bv}{\mathbf{v}}
\newcommand{\bw}{\mathbf{w}}

\newcommand{\p}{\partial}

\newcommand{\norm}[1]{\left\Vert#1\right\Vert}

\begin{document}
\baselineskip=18pt

\title[Stochastic MHD Equations]
{A well-posedness theory in Sobolev spaces for the stochastic magnetohydrodynamic equations in the whole space}

\author{Ildoo Kim \and Minsuk Yang}

\address{I. Kim: 
Department of mathematics, Korea University, 145 Anamro, Seongbukgu, Seoul, Republic of Korea}
\email{waldoo@korea.ac.kr}

\address{M. Yang: 
Department of Mathematics, Yonsei University, 50 Yonseiro Seodaemungu, Seoul, Republic of Korea}
\email{m.yang@yonsei.ac.kr}

\begin{abstract}
We prove the existence of a mild solution to the three dimensional incompressible stochastic magnetohydrodynamic equations in the whole space with the initial data which belong to the Sobolev spaces.
\end{abstract}

\maketitle

\section{Introduction}
\label{S1}

Magnetohydrodynamics is the branch of physics that studies the dynamics for electrically conducting fluids influenced by a magnetic field.
They are frequently generated in nature, for example, the sun, beneath the Earth's mantle, plasmas in space, liquid metals, and so on. 
For the background knowledge we refer the reader to Davidson's monograph \cite{Davidson}.
The aim of this paper is to establish the existence of a stochastic mild solution to the three dimensional incompressible magnetohydrodynamic (MHD) equations driven by stochastic external forces, which can be described as the following system of stochastic partial differential equations 
\begin{equation}
\label{E11}
\left\{
\begin{split}
&(\p_t - \Delta) \bu + \nabla\cdot(\bu\otimes\bu-\bb\otimes\bb) + \nabla \pi = g^k_1 \frac{dw^k}{dt} \\
&(\p_t - \Delta) \bb + \nabla\cdot(\bu\otimes\bb-\bb\otimes\bu) = g^k_2 \frac{dw^k}{dt} \\
&\nabla \cdot \bu = \nabla \cdot \bb = 0
\end{split}
\right.
\end{equation}
in $(0,T) \times \fR^3 $ with divergence-free initial vector fields $\bu_0$ and $\bb_0$, where $\bu$, $\bb$, $\pi$, and $\frac{dw^k}{dt}$ denote the velocity field, the magnetic field, the pressure of fluid, and  independent one-dimensional white noises $(k=1,2,\ldots)$, respectively.
We note that Einstein's summation convention is used throughout the paper. 

In the absence of magnetic field, the MHD equations reduce to the Navier--Stokes equations. 
There are huge literature about the theory of the Navier--Stokes equations.
For the deterministic Navier--Stokes equations, Fujita and Kato \cite{fujita1964navier} initiated the study for the existence of the mild solution with initial data in the critical Sobolev space $\dot{H}_2^{1/2}$.
Many mathematicians have been interested in its stochastic versions due to the complicate dynamics of fluid motions.
Naturally, there are many articles handling stochastic Naveri--Stokes equations extend the deterministic results.
However, there are only a few results for the stochastic MHD equations.
We refer the reader for background and history of these results to the introduction of \cite{Tan2016}, where the authors also considered  well-posedness of three-dimensional incompressible MHD equations with stochastic external forces. 

Before stating our main results, we motivate the definition of a mild solution and simplify the settings by introducing a few notations.
We denote by $\bP$ the Leray projection operator onto divergence-free vector fields.
In $\fR^3$ it can be expressed as
\[
\bP = I + \cR \otimes \cR,
\]
where $\cR = (\cR_1, \cR_2, \cR_3)$ denotes the Riesz transforms and $\left( (\cR \otimes \cR)\bu\right)^i = \sum_{j}\cR^i \cR^j u^j$.
By applying the Leray projection $\bP$ to \eqref{E11} formally, one can remove  the pressure term since $\bP \nabla \pi=0$.
By introducing new variables 
\begin{equation}
\label{E15}
\bv = \bu + \bb, \quad \bw = \bu - \bb,
\end{equation}
one can rewrite the equations \eqref{E11} as 
\begin{equation}
\label{pro eqn}
\left\{
\begin{split}
&(\p_t - \Delta) \bv =  - \bP\nabla\cdot(\bw\otimes\bv) + G^k_1 \frac{dw^k}{dt} \\
&(\p_t - \Delta) \bw = - \bP\nabla\cdot(\bv\otimes\bw) + G^k_2 \frac{dw^k}{dt} \\
&\nabla \cdot \bv = \nabla \cdot \bw = 0
\end{split}
\right.
\end{equation}
with the corresponding divergence-free initial data $\bv_0=\bu_0+\bb_0$ and $\bw_0=\bu_0-\bb_0$
and the stochastic external forces $G^k_1 = \bP g^k_1 + \bP g^k_2$ and $G^k_2 = \bP g^k_1 - \bP g^k_2$.

In order to neatly write the equations, we denote the heat semi-group by $S(t)=e^{t\Delta}$ and define the following bilinear map
\begin{equation}
\label{E14}
B(\bu,\bv)(t)=\int_0^t S(t-s) \bP \nabla \cdot (\bu \otimes \bv)(s) ds.
\end{equation}
Solving the heat equation by Duhamel's formula motivates
the definition of a mild solution of the stochastic MHD equations.
We say that $(\bv, \bw)$ is a (mild) solution to \eqref{pro eqn} on $(0,T)$ if it solves for $0\le t < T$ the integral equations
\begin{equation}
\label{E24}
\begin{split}
\bv & = S(t)\bv_0-B(\bw,\bv) +\int_0^t S(t-s) G^k_1(s)dw_s^k \\
\bw & = S(t)\bw_0 - B(\bv,\bw)+\int_0^t S(t-s) G^k_1(s)dw_s^k,
\end{split}
\end{equation}
where $w^k_t$ is the Brownian motion related to the white noisy $\frac{dw^k}{dt}$.

In \cite{Tan2016}, well-posedness for the stochastic MHD equations was studied for the initial data 
\[
\bv_0,\bw_0 \in  L_2 \left( \Omega, \rF_0; \dot{H}_{2,\sigma}^{1/2+\alpha}(\fR^3) \right)
\]
with $0 < \alpha < 1/2$.
We extend the well-posedness result for wider class of initial data 
\[
\bv_0,\bw_0 \in  L_5 \left( \Omega, \rF_0; \dot{H}^{-2/5}_p\right),
\]
but we admit that this paper does not cover the multiplicative noise case handled in \cite{Tan2016} with extra regularity condition on initial value.
The precise statements of our main results is the following two theorems.
The exact notations and definitions for function spaces are presented in the next section.

\begin{theorem}
\label{main thm 1} 
Let $\bv_0,\bw_0 \in  L_2 \left( \Omega, \rF_0; \dot{H}_{2,\sigma}^{1/2}(\fR^3) \right)$ and $G_1, G_2 \in  \dot {\bH}_{p,\sigma}^{1/2}(\infty, l_2)$.
\begin{enumerate}[(i)]
\item 
There exists a positive number $T_0$ such that the equation \eqref{pro eqn} with $T=T_0$ has a  solution 
\begin{align*}
\bv, \bw \in L_2 \left( \Omega,  \rF ; C\left([0,T_0]; \dot{H}_2^{1/2} \right)  \right)\cap L_2 \left(  \Omega \times (0,T_0), \bar \cP ; \dot{H}_2^{3/2}\right).
\end{align*}
\item
There exists a positive number $\varepsilon$ such that if
\[
\|(\bv_0,\bw_0)\|_{L_2 \left( \Omega, \rF_0; \dot{H}_{2}^{1/2} \right)} 
+\|(G_1,G_2)\|_{\dot {\bH}_{p}^{1/2}(\infty, l_2)} 
< \epsilon,
\]
then there exists a global in time solution $(\bv,\bw)$ to the equation \eqref{pro eqn} with 
\begin{align*}
\bv, \bw \in L_2 \left( \Omega,  \rF ; C\left([0,\infty); \dot{H}_2^{1/2} \right)  \right)\cap L_2 \left(  \Omega \times (0,\infty), \bar \cP ; \dot{H}_2^{3/2}\right).
\end{align*}
\end{enumerate}
\end{theorem}

\begin{theorem}
\label{main thm 3} 
Let $\bv_0,\bw_0 \in  L_5 \left( \Omega, \rF_0; \dot{H}^{-2/5}_{5,\sigma}\right)$ and $G_1, G_2 \in  \dot {\bH}_{5,\sigma}^{-1}(\infty, l_2)$.
\begin{enumerate}[(i)]
\item
There exists a positive number $T_0$ such that the equation  \eqref{pro eqn}  with $T=T_0$ has a solution $(\bv,\bw) \in \bL_5(T_0) \times \bL_5(T_0)$.
\item
There exists a positive number $\epsilon$ such that if
\[
\|(\bv_0,\bw_0)\|_{L_5 \left( \Omega, \rF_0; \dot{H}^{-2/5}_{5}\right)} 
+\|(G_1,G_2)\|_{\dot {\bH}_{5}^{-1}(\infty, l_2)} 
< \epsilon,
\]
then  there exists a global in time solution $(\bv,\bw)$ to the equation \eqref{pro eqn} with 
\[
\bv, \bw \in \bL_{5}( \infty).
\]
\end{enumerate}
\end{theorem}

\begin{remark}
If $\bX = \bL_5(\infty)$ or
\[
\bX=
L_4 \left( \Omega, \rF ; L_4((0,\infty) ; \dot{H}_2^{1/2 }(\R)) \right),
\]
then there is a  positive constant $C$ such that 
\[
\| B(\bv, \bw)\|_{\bX \times \bX} \leq C \| \bv\|_{\bX} \|\bw\|_{\bX} \qquad \forall \bv, \bw \in \bX,
\]
which is proved in Proposition \ref{bi est 1} and Proposition \ref{bi est 2}.
We can obtain an upper bound $\varepsilon < \frac{1}{4C_i C}$ $(i=5, 6)$ in the second parts of main theorems, where constants $C_5$ and $C_6$ appear in Corollary \ref{sto heat cor 1} and Corollary \ref{sto heat cor 2}, respectively.
Moreover, in this case, by the uniqueness of the fixed point theorem, the solution is unique in the closed subspace $\{ (\bu,\bv) \in \bX \times \bX : \|(\bu,\bv)\|_{\bX \times \bX} \leq 2\varepsilon\}$
(see Lemma \ref{cont lem}).
\end{remark}

\begin{remark}
Although we took physical constants to be 1, it is possible to consider the general physical constants in  the MHD equations, that  is,
\[
\left\{
\begin{split}
&(\p_t - \nu_1\Delta) \bv =  - \bP\nabla\cdot(\bw\otimes\bv) + G^k_1 \frac{dw^k}{dt} \\
&(\p_t - \nu_2 \Delta) \bw = - \bP\nabla\cdot(\bv\otimes\bw) + G^k_2 \frac{dw^k}{dt} \\
&\nabla \cdot \bv = \nabla \cdot \bw = 0,
\end{split}
\right.
\]
where the kinematic viscosity $\nu_1$ and the magnetic resistivity $\nu_2$ are positive constants. 
In this paper, we focus on studying \eqref{pro eqn} for simplicity.
\end{remark}

The organization of  this paper is as follows.
In Section 2, we introduce notations and definitions used throughout this paper. 
In Section 3, we  survey  linear theories for stochastic partial differential equations.
In Section 4, we prove bilinear estimates which play a crucial role.
In Section 5, we complete the proofs of main theorems.

\section{Notations and definitions}

The purpose of this section is to introduce notations and definitions which will be used throughout this paper.

\begin{itemize}
\item 
Let $\bN$ and $\bZ$ denote the natural number system and the integer number system, respectively.
As usual $\fR^{d}$, $d \in \bN$, stands for the Euclidean space of points $x=(x^{1},...,x^{d})$.
\item 
The gradient of a function $f$ is denoted  by 
\[
\nabla f = (D_1f, D_2f, \cdots, D_df).
\]
where $D_{i}f = \frac{\partial f}{\partial x^{i}}$ for $i=1,...,d$ and the divergence of a vector field $\bv=(\bv^1,\ldots, \bv^d)$ is denoted by
$$
\nabla \cdot \bv :=  \sum_{i=1}^d D_i \bv^i .
$$
\item 
Let $C^\infty(\fR^d)$ denote the space of infinitely differentiable functions on $\fR^d$. 
Let $C_c^\infty(\fR^d)$ denote the subspace of $C^\infty(\fR^d)$ with the compact support.
Let $C_{c,\sigma}^\infty(\fR^d)$ denote the subspace of $C_c^\infty(\fR^d)$ with divergence free.
Let $\cS(\fR^d)$ be the Schwartz space consisting of infinitely differentiable and rapidly decreasing functions on $\fR^d$.
We simply write $C^\infty, C_c^\infty, C_{c,\sigma}^\infty, \cS$ by omitting $(\fR^d)$.
\item  
For $\cO\subset \fR^d$ and a normed space $F$, we denote by $C(\cO;F)$ the space of all $F$-valued continuous functions $u : \cO \to F$ with the norm 
\[
|u|_{C}:=\sup_{x\in O}|u(x)|_F<\infty.
\]
\item 
For $p \in [1,\infty)$, a normed space $F$, and a  measure space $(X,\mathcal{M},\mu)$, we denote by $L_{p}(X,\cM,\mu;F)$ the space of all $\mathcal{M}^{\mu}$-measurable functions $u : X \to F$ with the norm 
\[
\left\Vert u\right\Vert _{L_{p}(X,\cM,\mu;F)}:=\left(\int_{X}\left\Vert u(x)\right\Vert _{F}^{p}\mu(dx)\right)^{1/p}<\infty
\]
where $\mathcal{M}^{\mu}$ denotes the completion of $\cM$ with respect to the measure $\mu$. 
If there is no confusion for the given measure and $\sigma$-algebra, we usually omit them.
\item 
We denote by $|\cO|$ the Lebesgue measure of a measurable set $\cO \subset \fR^d$.
\item 
Let $(\Omega,\rF, P)$ be a probability space and let $u(\omega, t,x)$ and $v(\omega, t,x)$ be stochastic processes on $\Omega \times (0,\infty) \times \fR^d$.
We say that with probability one, for all $t \in (0,\infty)$
\[
u(\omega,t,x) = v(\omega,t,x) \quad (x \text{-} a.e.)
\]
if there exists $\Omega' \subset \Omega$ such that $P(\Omega')=1$ and for all $(\omega' , t) \in \Omega' \times (0,\infty)$,
\[
u(\omega',t,x)= v(\omega',t,x)
\]
holds for almost every $x \in \fR^d$. 
For the notational convenience, the random parameter $\omega$ will be usually omitted. 
\item 
We denote the $d$-dimensional Fourier transform of $f$ by 
\[
\cF[f](\xi) := \int_{\fR^{d}} e^{-2\pi i \xi \cdot x} f(x) dx
\]
and the $d$-dimensional inverse Fourier transform of $f$ by 
\[
\cF^{-1}[f](x) := \int_{\fR^{d}} e^{ 2\pi ix \cdot \xi} f(\xi) d\xi.
\]
\item 
If we write $C=C(a,b,\cdots)$, this means that the constant $C$ depends only on $a,b,\cdots$. 
\item 
We shall write $A \lesssim B$ if there is a positive generic constant $C$ such that $|A| \le C|B|$.
\end{itemize}

Let $(\Omega,\rF,P)$ be a complete probability space, and $\{\rF_{t},t\geq0\}$ be a filtration satisfying the usual condition, i.e. $\{\rF_t\}$ is increasing , right continuous, and each  $\rF_t$ contains all
$(\rF,P)$-null sets.
In other words, $\rF_{t_1} \subset \rF_{t_2}$ if $t_1 \leq t_2$, $ \bigcap_{t<s}\rF_s = \rF_t$, and $A \subset \rF_t$ for all $t \geq 0$ if there exists a $B \in \rF$ such that $A \subset B$ and $P(B)=0$.  
We denote by $\cP$ the predictable $\sigma$-field generated by $\{\rF_{t},t\geq0\}$.
We assume that $w^k_{t}$ are independent one-dimensional Brownian motions (Wiener
processes) on $(\Omega,\rF,P)$ for $(k=1,2,\ldots)$ and they are relative to $\{\rF_{t},t\geq0\}$.

We end this section by introducing inhomogeneous and homogeneous Sobolev spaces and related function spaces used in this article.

\begin{definition}[Inhomogeneous Sobolev spaces]
\begin{enumerate}
\item
For $n \in \fR$ and $p \in (1,\infty)$, define the space 
$H^{n}_p(\fR^3)=(1-\Delta)^{-n/2}L_p(\fR^3)$ (called
the space of Bessel potentials or the Sobolev space with fractional
derivatives) as the set of all tempered distributions $u$ such that
\[
(1-\Delta)^{n/2}u := \cF^{-1}\left[\left(1+|2\pi \xi|^2\right)^{n/2}\cF(u)(\xi) \right] \in L_p(\fR^3)
\]
with the norm
\[
\|u\|_{H^{n}_p(\fR^3)} :=\|(1-\Delta)^{n/2}u\|_{L_p(\fR^3)} < \infty.
\]
\item
The set of all $u =(u^1, u^2, u^3 )$ such that $u^1,u^2,u^3 \in H_p^n( \fR^3)$ is denoted by
$H^{n}_{p}(\fR^3; \fR^3 )$
and the norm is given by
\[
\|u\|_{H^{n}_{p}(\fR^3; \fR^3 )}
:= \sum_{i=1}^3 \left\| ( 1-\Delta)^{n/2}u^i \right\|_{L_p(\fR^3)} < \infty.
\]
$H^{n}_{p,\sigma}(\fR^3; \fR^3 )$ denotes the closure of $C_{c,\sigma}^\infty(\fR^3; \fR^3 )$ in $H^{n}_{p}(\fR^3; \fR^3 )$.
\end{enumerate}
\end{definition}

For a sequence
$$
a=(a^1,a^2,\ldots)=(a^k)_{k \in \fN},
$$
we define $\|a\|^2_{l_2} :=  \sum_{k=1}^\infty |a^k|^2$ and denote by $l_2$ the space of all sequences $a$ so that $\|a\|_{l_2} < \infty$.
For $u=(u^k)_{k \in \fN}$ $(u^k \in H_p^n(\fR^3))$ , we define the space $H^n_p(\fR^3;l_2)$ as the set of all $l_2$-valued tempered distributions such that 
$$
\|u\|_{H^n_p(\fR^3;l_2)}
:=\left\| |(1-\Delta)^{n/2}u|_{l_2} \right\|_{L_p(\fR^3)} 
< \infty.
$$
The set of all $u =(u^1, u^2, u^3 )$ such that $u^1,u^2,u^3 \in H_p^n( \fR^3;l_2)$ is denoted by
$H^{n}_{p}(\fR^3; \fR^3 \times l_2)$
with the norm 
$$
\|u\|_{H^{n}_{p}(\fR^3; \fR^3  \times l_2)}
:= \sum_{i=1}^3 \left\|  u^i  \right\|_{H_p^n(\fR^3; l_2)}
$$
$H^{n}_{p,\sigma}(\fR^3; \fR^3  \times l_2)$ denotes the subspace of $H^{n}_{p}(\fR^3; \fR^3 \times l_2)$ in which every component  is divergence-free, i.e.
\begin{align*}
u=\left(u^{ik} ~(i =1,2,3, ~k \in \bN) \right)~\in H^{n}_{p,\sigma}(\fR^3; \fR^3\times l_2 ) \quad 
\end{align*}
if and only if $u^{ik} \in H^{n}_{p,\sigma}(\fR^3; \fR^3 )$ for all $i=1,2,3$ and $k \in \bN$.
In particular, we put $L_{p,\sigma}:=H^{0}_{p,\sigma}$.
For the notational convenience, we set
\begin{align*}
H_p^n &= H^{n}_{p}(\fR^3; \fR^3 ), \\
H^{n}_{p,\sigma} &= H^{n}_{p,\sigma}(\fR^3; \fR^3), \\
H_p^n(l_2) &= H^{n}_{p}(\fR^3; \fR^3 \times l_2), \\
H^{n}_{p,\sigma}(l_2) &= H^{n}_{p,\sigma}(\fR^3; \fR^3 \times l_2).
\end{align*}
We write $u\in \bH^{n}_p(T)$ if $u$ is an $H^{n}_p$-valued $\bar{\cP}$-measurable process satisfying
\[
\|u\|_{\bH^{n}_p(T)}:= \left( \bE \left[ \int^T_0
\,\|u\|^p_{H^{n}_p}\,dt \right] \right)^{1/p}<\infty.
\]

For the notational convenience, we set
\begin{align*}
\bH^{n}_{p}(T) &:= L_p(\Omega\times (0,T),\bar \cP;H^{n}_{p}), \\
\bH^{n}_{p,\sigma}(T) &:= L_p(\Omega\times (0,T),\bar \cP;H^{n}_{p,\sigma}) \\
\bH^{n}_{p}(T,l_2) &:= L_p(\Omega\times (0,T),\bar \cP;H^{n}_{p}(l_2)), \\ 
\bH^{n}_{p,\sigma}(T,l_2) &:= L_p(\Omega\times (0,T),\bar \cP;H^{n}_{p,\sigma}(l_2)).
\end{align*}
Similarly, we write 
\begin{align*}
U^{n}_{p} &:= L_p(\Omega, \rF_0;H^{n}_{p}), \\ 
U^{n}_{p,\sigma} &:= L_p(\Omega, \rF_0;H^{n}_{p,\sigma}),
\end{align*}
and simply
\begin{align*}
\bL_{p}(T) &:=\bH^{0}_{p}(T), \\
\bL_{p,\sigma}(T) &:=\bH^{0}_{p,\sigma}(T), \\
\bL_{p}(T,l_2) &:=\bH^{0}_{p}(T,l_2), \\
\bL_{p,\sigma}(T,l_2) &:=\bH^{0}_{p,\sigma}(T,l_2).
\end{align*}

\begin{definition}[Homogeneous Sobolev spaces]
\begin{enumerate}
\item
We denote by $\dot{H}_p^{n}(\fR^3; \fR^3)$ the space of all $\fR^3$-valued tempered distribution $u=(u^1,u^2,u^3)$ modulo by polynomials such that 
\[
(-\Delta)^{n/2} u^i := \cF^{-1} \left[ |2\pi \xi|^{n}  \cF(u^i)(\xi)  \right] \in L_p(\fR^3)
\]
with the norm
\[
\|u\|_{\dot {H_p^{n} } (\fR^3;\fR^3)}
:= \sum_{i=1}^3 \left\| (-\Delta)^{n/2} u^i   \right\|_{L_p(\fR^3)} <\infty.
\]
We denote by $\dot{H}_p^{n}(\fR^3; \fR^3 \times l_2)$ the space of all sequence $u=(u^{ik})_{i=1,2,3, ~k \in \fN}$ such that $u^k \in \dot{H_p^n}(\fR^3; \fR^3)$ for all $k \in \bN$ and
$$
\|u\|_{\dot{H_p^n}(\fR^3; \fR^3 \times l_2)}:=  \sum_{i=1}^3 \left\|   \left(\sum_{k=1}^\infty\left|(-\Delta)^{n/2} u^{ik} \right|^2  \right)^{1/2} \right\|_{L_p(\fR^3)}  < \infty.
$$
$\dot{H}_p^{n, \sigma}(\fR^3; \fR^3)$ and $\dot{H}_p^{n, \sigma}(\fR^3; \fR^3 \times l_2)$ denote the subspace of
$\dot{H}_p^{n}(\fR^3; \fR^3)$ and $\dot{H}_p^{n}(\fR^3; \fR^3 \times l_2)$ whose elements are divergence free, respectively, i.e.
$\dot{H}_p^{n, \sigma}(\fR^3; \fR^3)$ and $\dot{H}_p^{n, \sigma}(\fR^3; \fR^3 \times l_2)$ are closures of $C_c^\infty$ with respect the norms 
in $\dot{H}_p^{n}(\fR^3; \fR^3)$ and $\dot{H}_p^{n}(\fR^3; \fR^3 \times l_2)$, respectively.

\item
Similar to inhomogeneous function spaces, we set deterministic spaces
\begin{align*}
\dot H_p^n &= H^{n}_{p}(\fR^3; \fR^3) \\ 
\dot H^{n}_{p,\sigma} &= \dot H^{n}_{p,\sigma}(\fR^3; \fR^3) \\
\dot H_p^n(l_2) &= H^{n}_{p}(\fR^3; \fR^3 \times l_2) \\
\dot H^{n}_{p,\sigma}(l_2) &= \dot H^{n}_{p,\sigma}(\fR^3; \fR^3 \times l_2),
\end{align*}
and  stochastic spaces 
\begin{align*}
\dot \bH^{n}_{p}(T) &:= L_p(\Omega\times (0,T),\bar \cP; \dot  H^{n}_{p}) \\
\dot \bH^{n}_{p,\sigma}(T) &:= L_p(\Omega\times (0,T),\bar \cP; \dot  H^{n}_{p,\sigma}) \\
\dot \bH^{n}_{p}(T,l_2) &:= L_p(\Omega\times (0,T),\bar \cP; \dot H^{n}_{p}(l_2)) \\
\dot \bH^{n}_{p,\sigma}(T,l_2) &:= L_p(\Omega\times (0,T),\bar \cP; \dot H^{n}_{p,\sigma}(l_2)),
\end{align*}
and
\begin{align*}
\dot  U^{n}_{p} &:= L_p(\Omega, \rF_0; \dot H^{n}_{p}) \\
\dot  U^{n}_{p} &:=L_p(\Omega, \rF_0; \dot H^{n}_{p,\sigma}).
\end{align*}
\end{enumerate}
\end{definition}

\begin{remark}
It is well-known that two norms $\|\cdot \|_{H_p^n}$ and $\|\cdot \|_{\dot H_p^n}+\|\cdot\|_{L_p}$ are equivalent if $1 < p < \infty$ and $n>0$ (cf. \cite[Theorem 6.3.2]{JLinter1976}).
Thus, $H_p^n=\dot H_p^n   \cap L_p$ and
$\|\cdot \|_{ \dot H_p^n} \lesssim   \|\cdot \|_{ H_p^n}$ if $1 < p <\infty$ and $n >0$.
\end{remark}

\section{Linear theories for stochastic heat equations}

In this section, we introduce linear theories for stochastic PDEs.
Recently, analytic regularity theories for stochastic PDEs have been well developed.
We refer the reader to \cite{Krylov1999}, which is considered as one of bibles in this area. 
However, most of the estimates and theories are handled in inhomogeneous Sobolev space setting.
To prove our main theorems, we need to obtain homogeneous type estimates for linear stochastic PDEs.
Since we could not find an appropriate reference to show homogeneous type estimates for stochastic PDEs, we  give detailed proofs.
In addition, we note that all functions in the following theorem are $\fR^3$-valued, but the results in \cite{Krylov1999} are for scalar-valued functions.
Since our leading operator is Laplacian, those are easily extended to $\fR^3$-valued functions without any difficulty. 

\begin{proposition}
				\label{linear thm}
Let $n \in \fR$ and $T \in (0,\infty)$.
If $u_0 \in U_p^{n+2-2/p}$, $f \in \bH_p^{n}(T)$, and $g \in \bH_p^{n+1}(T,l_2)$, then there exists a unique solution $u \in \bH_p^{n+2}(T)$ to 
\begin{equation}
\label{heat eqn}
\begin{split}
u_t(t,x) &= \Delta u(t,x) + f(t,x) + g^k(t,x) \frac{dw^k}{dt} \\
u(0,x) &= u_0(x)
\end{split}
\end{equation}
in $(0,T) \times \fR^3$ in the sense that for any $\phi \in C_c^\infty$, with probability one, for all $t \in (0,T)$,
\begin{align*}
\left(u(t,\cdot) , \phi \right) = \left( u_0, \phi \right) + \int_0^t \left[\left(  u(s,\cdot), \Delta \phi \right) + \left(f(s,\cdot) , \phi\right) \right] ds + \int_0^t \left( g^k(s,\cdot), \phi \right) dw_s^k.
\end{align*}
Moreover, there exist positive constants $C_1(T,p)$, $C_2(p)$, $C_3(T,p)$, and $C_4(p)$ such that
\begin{align}
						\label{heat est}
\|u\|_{\bH_p^{n+2}(T)} &\leq C_1 \left( \|u_0\|_{U_p^{n+2-2/p}} + \|f\|_{\bH_p^{n}(T)} + \|g\|_{\bH_p^{n+1}(T,l_2)}\right) \\
						\label{heat est 2}
\|  u\|_{\dot \bH_p^{n+2}(T)} &\leq C_2 \left( \|  u_0\|_{\dot U_p^{n+2-2/p}} + \|  f\|_{\dot \bH_p^{n}(T)} + \|g\|_{ \dot \bH_p^{n+1}(T,l_2)}\right)
\end{align}
and 
\begin{align}
\bE \sup_{0 \leq t \leq T}\|  u(t,\cdot)\|^p_{ H_p^{n+2 -2/p}}   
						\label{heat est 3}
&\leq C_3 \left( \|  u_0\|^p_{\dot U_p^{n+2-2/p}} + \|  f\|^p_{\bH_p^{n+1-2/p}(T)} + \|g\|^p_{ \bH_p^{n+2-2/p}(T,l_2)}\right) \\
\bE \sup_{0 \leq t \leq T}\|  u(t,\cdot)\|^p_{\dot H_p^{n+2 -2/p}}    
						\label{heat est 4}
&\leq C_4 \left( \|  u_0\|^p_{\dot U_p^{n+2-2/p}} + \|  f\|^p_{\dot \bH_p^{n+1-2/p}(T)} + \|g\|^p_{ \dot \bH_p^{n+2-2/p}(T,l_2)}\right).
\end{align}
\end{proposition}

\begin{proof}
This proposition with the estimates \eqref{heat est} and \eqref{heat est 3} was proved in \cite[Lemma 4.1 and Theorem 4.10]{Krylov1999}.
We only prove \eqref{heat est 2} since the proof of \eqref{heat est 4} is similar.
We may prove it with assuming $T=1$ due to scaling.
Indeed, Brownian motions have the self-similarity, that is, 
\[
\tilde w_t:=\frac{1}{\sqrt T} w_{Tt}.
\]
is also a Brownian motion.
If the proposition is proved for $T=1$, then we define 
\begin{align*}
\tilde u(t,x) &= u(Tt,\sqrt T x) \\
\tilde u_0 (x) &= u_0(\sqrt T x) \\
\tilde f(t,x) &= T f(Tt, \sqrt T x) \\
\tilde g (t,x) &=  \sqrt T g(Tt, \sqrt x),
\end{align*}
where $u,u_0,f,g$ satisfy \eqref{heat eqn} so that $\tilde u$ is a solution to 
\begin{align*}
&\tilde u_t(t,x)= \Delta \tilde u(t,x) + \tilde f(t,x) + \tilde g^k(t,x) \frac{d \tilde w^k}{dt} \\
&\tilde u(0,x)= \tilde u_0(x)
\end{align*}
in $(0,1) \times \fR^3$.
Due to the estimate for $\tilde u$, we have
\begin{align*}
\|  u\|_{\dot \bH_p^{n+2}(T)} 
&= T^{1/p} T^{-(n+2)/2} T^{3/(2p)}\|  \tilde u \|_{\dot \bH_p^{n+2}(1)}  \\
&\leq T^{1/p} T^{-(n+2)/2} T^{3/(2p)} C_2 \left( \|  \tilde u_0\|_{\dot U_p^{n+2-2/p}} + \|  \tilde f\|_{\dot \bH_p^{n}(1)} + \|\tilde g\|_{ \dot \bH_p^{n+1}(1,l_2)}\right) \\
&= C_2 \left( \|  u_0\|_{\dot U_p^{n+2-2/p}} + \|  f\|_{\dot \bH_p^{n}(T)} + \|g\|_{ \dot \bH_p^{n+1}(T,l_2)}\right).
\end{align*}

Now, we assume $T=1$ and split the proof into four cases.
\begin{itemize}
\item {\bf Case 1}. ($f=0$ and $g=0$). 
Then taking $(-\Delta)^{(n+2)/2-1/p}$ to both sides of \eqref{heat eqn} and applying  \eqref{heat est} with $n=-2+2/p$, we have
\begin{equation}
\label{0224 1}
\begin{split}
\|u\|_{\dot \bH_p^{n+2}} 
\lesssim \|\Delta^{(n+2)/2-1/p}u\|_{ \bH_p^{2/p}} 
\lesssim \|\Delta^{(n+2)/2-1/p} u_0\|_{ U_p^{0}} 
= \|u_0\|_{ \dot U_p^{n+1-2/p}}.
\end{split}
\end{equation}
\item {\bf Case 2}. ($u_0=0$ and $g=0$). 
Then taking $(-\Delta)^{n/2}$ to both sides of \eqref{heat eqn} and applying \eqref{heat est} with $n=0$, we have
\begin{equation}
\label{0224 2}
\|u\|_{\dot \bH_p^{n+2}} 
\lesssim \|\Delta^{n/2}u\|_{ \bH_p^{2}} 
\lesssim  \|\Delta^{n/2}f\|_{ \bL_p  }
= \| f\|_{ \dot \bH_p^n }.
\end{equation}
\item {\bf Case 3}.  ($u_0=0$ and $f=0$). 
Then taking $(-\Delta)^{(n+1)/2}$ to both sides of \eqref{heat eqn} and applying \eqref{heat est} with $n=-1$, we have
\begin{equation}
\label{0224 3}
\begin{split}
\|  u\|_{\dot \bH_p^{n+2}} 
\lesssim \|  \Delta^{(n+1)/2}u\|_{ \bH_p^{1}} 
\lesssim  \|   \Delta^{(n+1)/2} g\|_{ \bL_p( T,l_2)} 
= \| g\|_{ \dot \bH_p^{n+1} (T,l_2)}.
\end{split}
\end{equation}

\item {\bf Case 4}.(General Case).  
Suppose 
$u^1$ is a unique solution to 
\begin{align*}
&u^1_t(t,x)= \Delta u^1(t,x) \\
&u^1(0,x)= u_0(x),
\end{align*}
$u^2$ is a unique solution to 
\begin{align*}
&u^2_t(t,x)= \Delta u^2(t,x) + f(t,x) \\
&u^2(0,x)= 0,
\end{align*}
and $u^3$ is a unique solution to 
\begin{align*}
&u^3_t(t,x)= \Delta u^3(t,x)  + g^k(t,x) \frac{dw^k}{dt} \\
&u^3(0,x)= 0.
\end{align*}
Then we get the solution $u=u^1+u^2+u^3$ for the general case.
Combining the estimates \eqref{0224 1}, \eqref{0224 2}, and \eqref{0224 3}, we obtain that \eqref{heat est 2} with $T=1$.
\end{itemize}
\end{proof}

\begin{remark}
The constants $C_1$ and $C_3$ depend on $T$. However, one can take $C_1$ and $C_3$ uniformly for all $T \in [0,\bar{T}]$, i.e. for all $T \in [0,\bar T]$,
\begin{align*}
\|u\|_{\bH_p^{n+2}(T)} \leq C_1 \left( \|u_0\|_{U_p^{n+2-2/p}} + \|f\|_{\bH_p^{n}(T)} + \|g\|_{\bH_p^{n+1}(T,l_2)}\right),
\end{align*}
and
\begin{align*}
\bE \sup_{0 \leq t \leq T}\|  u(t,\cdot)\|^p_{ H_p^{n+2 -2/p}} 
& \leq C_3 \left( \|  u_0\|^p_{ U_p^{n+2-2/p}} + \|  f\|^p_{ \bH_p^{n+1-2/p}(T)} + \|g\|^p_{  \bH_p^{n+2-2/p}(T,l_2)}\right),
\end{align*}
where constants $C_1$ and $C_3$ depend only on $\bar T$ and $p$.
\end{remark}

We denote by $\bH_c^\infty(T, l_2)$ the space of stochastic processes $g=(g^1,g^2, \ldots)$ such that $g^k=0$ for all large $k$ and each $g^k$ is  of the type
$$
g^k(t,x)= \sum_{i=1}^{j(k)}1_{(\tau_{i-1},\tau_i]}(t) g^{ik}(x),
$$
where  $j(k) \in \bN$, $g^{ik} \in C_c^\infty$, and $\tau_i$ are stopping times with $\tau_i \leq T$.
Similarly, we denote by $\bH_c^\infty(T)$ the space of the processes $g$ such that 
$$
g(t,x)= \sum_{i=1}^{j}1_{(\tau_{i-1},\tau_i]}(t) g^{i}(x),
$$
where $j \in \bN$, $g^{i} \in C_c^\infty$, and $\tau_i$ are stopping times with $\tau_i \leq T$.
Lastly, we denote by $\bH_c^\infty(\fR^3)$ the space of the random variables $g_0$ of the type
$$
g_0(\omega,x)= 1_A(\omega)g(x)
$$
where  $g \in C_c^\infty$, and $ A \in \rF_0$.

\begin{remark}
				\label{main rem}
\begin{enumerate}[(i)]
\item It is known that $\bH^\infty_c(T,l_2)$ is dense in $\bH^n_p(T,l_2)$ for all $p \in (1,\infty)$ and $n \in \fR$ (for instance, see \cite[Theorem 3.10]{Krylov1999}).
In particular, $\bH_c^\infty(T)$ is dense in $\bH^n_p(T)$ for all $p \in (1,\infty)$ and $n \in \fR$.
Following the idea of \cite[Theorem 3.10]{Krylov1999}, one can also easily check that $\bH_c^\infty(\fR^3)$ is dense in $U_p^n$ for all $p \in (1,\infty)$ and $n \in \fR$.

\item If $g \in \bH_p^{n+1}(T,l_2)$ , then the stochastic integral 
\[
\int_0^t \left( g^k(s,\cdot), \phi \right) dw_s^k
\]
is well-defined in the classical It\^o-sense (cf. \cite[Remark 3.2]{Krylov1999}).
Moreover, the $H_p^{n+1}(l_2)$-valued stochastic integral
\[
\int_0^t g(s,x) dw_s
\]
is defined by recently developed UMD-space valued stochastic integral theory (cf. \cite{van2007stochastic}).
\item The dual space of $H_p^n$ is $H_q^{-n}$, where $q= \frac{p}{p-1}$, and one can find a countable subset of $C_c^\infty$ which is dense in $H_q^{-n}$.
Thus, \eqref{heat est} can be interpreted in strong sense. i.e. \eqref{heat eqn} holds if and only if with probability one, for all $t \in (0,T)$,
\begin{align*}
u(t,x)  =  u_0 + \int_0^t \left(\Delta u(s,x) + f(s,x) \right) ds + \int_0^t g^k(s,x) dw_s^k
\end{align*}
where the equality holds as an element of $H_p^n$,
$\int_0^t \left(\Delta u(s,x) + f(s,x) \right) ds$
is $H_p^n$-valued classical Bochner's integral, and 
$\int_0^t g(s,x) dw_s$ is $H_p^n(l_2)$-valued stochastic integral.

\item
If $u_0 \in \bH_c^\infty(\fR^3)$, $f \in \bH_c^\infty(T)$, and $g \in \bH_c^\infty(T,l_2)$, then the solution $u$ is given by (cf. \cite[proof of Theorem 4.2]{Krylov1999})
\begin{align}
						\label{repre eq}
u(t,x) = S(t)u_0(x) + \int_0^t S(t-s)f(s,\cdot)(x) ds + \int_0^t S(t-s)g(s,\cdot)dw^k_s,
\end{align}
where 
\begin{align*}
p(t,x) &=  (4\pi t)^{-3/2} \exp \left( - |x|^2/(4t)  \right) \\
S(t)u_0(x) &=  \int_{\fR^3}p(t,x-y) u_0(y)dy.
\end{align*}
Due to \eqref{heat est}, the standard approximation, and UMD space-valued stochastic integration theories, 
\eqref{repre eq} holds even for general 
$u_0 \in \bH_p^{n+2}(\fR^3)$, $f \in \bH_p^n(T)$, and $g \in \bH_p^n(T,l_2)$.

\item 
In Proposition \ref{linear thm}, we assumed that  $u_0 \in U_p^{n+2-2/p}$, $f \in \bH_p^{n}(T)$, $g \in \bH_p^{n+1}(T,l_2)$.
However by using approximations in $u^N_0 \in \dot U_p^{n+2-2/p} \cap U_p^{n+2-2/p}$, $f^N \in \dot \bH_p^{n}(T) \cap \bH_p^{n}(T)$, $g^N \in  \dot \bH_p^{n+1}(T,l_2) \cap \bH_p^{n+1}(T,l_2)$ $(N=1,2,\ldots)$
or applying a UMD-space valued stochastic maximal $L_p$-inequality (\cite[Theorem 1.1]{van2012stochastic}), one can easily check that 
\eqref{heat est 2} and \eqref{heat est 4} hold for all $u_0 \in \dot{U}_p^{n+2-2/p}$, $f \in \dot{\bH}_p^{n}(T)$, $g \in \dot{\bH}_p^{n+1}(T,l_2)$, and $u$ defined as in \eqref{repre eq}.
\end{enumerate}
\end{remark}

We state a few corollaries, which is easily deduced by Proposition \ref{linear thm} and the representation \eqref{repre eq}.
\begin{corollary}
					\label{sto heat cor 1}
There exists a positive constant $C_5$ such that for all $T \in (0,\infty)$, $u_0 \in \dot U_p^{1/2}$, and $g \in \dot \bH_2^{1/2}(T,l_2)$,
\begin{align*}
&\left\|  S(t)u_0(x)  + \int_0^t S(t-s)g(s,\cdot)dw^k_s \right\|_{\dot \bH_2^{3/2}(T)}  \\
&+ \bE \sup_{0 \leq t \leq T}\left\| S(t)u_0(x)  + \int_0^t S(t-s)g(s,\cdot)dw^k_s  \right\|^p_{\dot H_2^{1/2}}  \\
&\leq C_5 \left( \|  u_0\|_{\dot U_2^{1/2}}  + \|g\|_{ \dot \bH_2^{1/2}(T,l_2)}\right).
\end{align*}
\end{corollary}

\begin{corollary}
					\label{sto heat cor 2}
There exists a positive constant $C_6$ such that for all $T \in (0,\infty)$, $u_0 \in \dot U_5^{-2/5}$, 
and $g \in \dot \bH_5^{-1}(T,l_2)$,
\begin{align*}
\left\|  S(t)u_0(x)+ \int_0^t S(t-s)g(s,\cdot)dw^k_s \right\|_{\dot \bL_5(T)}  
&\leq C_6 \left( \|  u_0\|_{\dot U_5^{-2/5}} + \|g\|_{ \dot \bH_5^{-1}(T,l_2)}\right).
\end{align*}
\end{corollary}

\begin{corollary}
					\label{heat cor}
(i) There exists a positive constants $C_7$ such that for all $T \in (0,\infty)$ and $u_0 \in \dot H_2^{1/2}$,
\begin{align*}
\int_0^T\|S(t)u_0 \|^2_{ \dot H_2^{3/2}} dt \leq C_7 \|u_0\|_{\dot H_2^{1/2}}.
\end{align*}

(ii) There exists a positive constant $C_8$ such that for all $T \in (0,\infty)$ and $f \in L^2((0,T) ; \dot{H}_2^{-1/2})$,
\begin{align*}
& \sup_{ 0\leq t \leq T}\left\| \int_0^tS(t-s)f(s, \cdot)  ds  \right\|^2_{ \dot H_2^{1/2}}
+\int_0^T\left\| \int_0^tS(t-s)f(s, \cdot)  ds  \right\|^2_{ \dot H_2^{3/2}} dt \\
&\quad \leq C_8 \int_0^T\|f(t, \cdot)\|^2_{\dot H_2^{-1/2}}dt.
\end{align*}

(iii) There exists a positive constant $C_9$ such that for all $T \in (0,\infty)$ and $u_0 \in H_5^{-2/5}$,
\begin{align*}
\|S(t)u_0 \|_{L_5((0,T) \times \fR^3)} \leq C_9 \|u_0\|_{H^{-2/5}_5}.
\end{align*}
\end{corollary}

\section{Bilinear Estimates}
\label{S3}

The goal of this section is to prove bilinear estimates, Proposition \ref{bi est 1} and Proposition \ref{bi est 2}, which play key roles to obtain our existence results for stochastic MHD equations.

Before going further, we recall that if
\[
\bv \in C([0,T];\dot{H}_2^{1/2}) \cap L^2((0,T);\dot{H}_2^{3/2}),
\]
then from an interpolation inequality in the Sobolev space,
\begin{equation}
						\label{E31}
\norm{\bv}_{L^4((0,T);\dot{H}_2^{1})} \lesssim \norm{\bv}_{C([0,T];\dot{H}_2^{1/2})}^{1/2} \norm{\bv}_{L^2((0,T);\dot{H}_2^{3/2})}^{1/2}
\end{equation}
and hence
\[
C([0,T];\dot{H}_2^{1/2}) \cap L_2((0,T);\dot{H}_2^{3/2}) \subset L_4((0,T);\dot{H}_2^{1}).
\]
For the notational convenience, we set 
\[
X_1 = C([0,T]  ; \dot{H}_2^{1/2}) \cap L^2((0,T) ; \dot{H}_2^{3/2 }) 
\]
with the norm
\[
\|\cdot\|_{X_1} = \|\cdot \|_{C([0,T];\dot{H}_2^{1/2})}+ \|\cdot\|_{L_2((0,T);\dot{H}_2^{3/2})}.
\]

For the corresponding stochastic function spaces, we set 
\begin{align*}
{\bX_{1,T}} &:=  L_2 \left( \Omega, \rF ; C([0,T] ; \dot{H}_2^{1/2 }(\R)) \right) \cap \dot\bH^{3/2}_2(T), \\
{\bX_{2,T}} &:=  L_4 \left( \Omega, \rF ; L_4([0,T] ; \dot{H}_2^{1/2 }(\R)) \right) 
\end{align*}
with the norms
\begin{align*}
\|\cdot\|_{\bX_{1,T}} &= \|\cdot \|_{L_2 \left( \Omega, \rF ; C([0,T]:\dot{H}_2^{1/2}(\R)) \right)}+ \|\cdot\|_{\dot\bH^{3/2}_2(T)}, \\
\|\cdot\|_{\bX_{2,T}} &=  \left( \bE \int_0^T \norm{\cdot}_{\dot{H}_2^1}^4 dt\right)^{1/4}.
\end{align*}
Due to \eqref{E31}, we have $\bX_{1,T} \subset \bX_{2,T}$.

\begin{proposition}
					\label{bi est 1}
There exists a constant $C_{10}>0$ such that for all $T>0$ and $\bv, \bw \in {\bX_{1,T}}$,
\begin{align*}
\norm{B(\bw,\bv)}_{{\bX_{1,T}}}  
&\leq C_{10}  \|\bv \|_{\bX_{2,T}} \|\bw\|_{\bX_{2,T}}.
\end{align*}
\end{proposition}

\begin{proof}
We begin by recalling 
\[
B(\bu,\bv)(t)=\int_0^t S(t-s) \bP \nabla \cdot (\bu \otimes \bv)(s) ds.
\]
Since the embedding $L_{3/2} \subset \dot{H}_2^{-1/2}$ is continuous and the projection operator $\bP$ is continuous on $L_{3/2}$,
we have 
\begin{align*}
\norm{\bP \nabla \cdot (\bv \otimes \bw)}_{\dot{H}_2^{-1/2}} 
\lesssim \norm{\bP \nabla \cdot (\bv \otimes \bw)}_{L_{3/2}}
\lesssim \norm{\nabla \cdot (\bv \otimes \bw)}_{L_{3/2}}. 
\end{align*}
Since $\nabla \cdot \bv=0$, we use the H\"older inequality and the Sobolev inequality to get 
\[
\norm{\nabla \cdot (\bv \otimes \bw)}_{L_{3/2}} 
\lesssim \norm{\bv \cdot \nabla \bw}_{L_{3/2}} 
\lesssim \norm{\bv}_{L_{6}} \norm{\nabla \bw}_{L_2}.
\lesssim \norm{\bv}_{\dot{H}_2^1} \norm{\bw}_{\dot{H}_2^1}
\]
Thus, we use the Cauchy--Schwarz inequality inequality to get 
\begin{equation}
							\label{2018-07-16-1}
\begin{split}
&\int_0^T \norm{\bP \nabla \cdot (\bv \otimes \bw)}_{\dot{H}_2^{-1/2}}^2 dt \\
&\lesssim \int_0^T \norm{\bv}_{\dot{H}_2^1}^2 \norm{\bw}_{\dot{H}_2^1}^2 dt  \\
&\lesssim \left(\int_0^T \norm{\bv}_{\dot{H}_2^1}^4 dt\right)^{1/2}  \left(\int_0^T \norm{\bw}_{\dot{H}_2^1}^4 dt\right)^{1/2}.
\end{split}
\end{equation}
If we put 
\[
f := \bP \nabla \cdot (\bv \otimes \bw) \in L_2((0,T):\dot{H}_2^{-1/2}(\R)),
\]
then from Corollary \ref{heat cor}(ii), 
\begin{equation}
\begin{split}
							\label{225 2}
& \sup_{ 0\leq t \leq T}\left\| \int_0^tS(t-s)f(s, \cdot)  ds  \right\|^2_{ \dot H_2^{1/2}}
+\int_0^T\left\| \int_0^tS(t-s)f(s, \cdot)  ds  \right\|^2_{ \dot H_2^{3/2}} dt \\
&\lesssim \int_0^T\|f(t, \cdot)\|^2_{\dot H_2^{-1/2}}dt.
\end{split}
\end{equation}
Combining \eqref{2018-07-16-1} and \eqref{225 2}, we obtain for all $T >0$ and $\bv,\bw \in L^4((0,T) ; \dot{H}_2^{1})$,
\begin{align*}
\norm{B(\bv,\bw)}_{X_1}^2
&\lesssim \left(\int_0^T \norm{\bv}_{\dot{H}_2^1}^4 dt\right)^{1/2}  \left(\int_0^T \norm{\bw}_{\dot{H}_2^1}^4 dt\right)^{1/2}.
\end{align*}
Thus, we use the Cauchy--Schwarz inequality inequality to obtain that 
\begin{align*}
\norm{B(\bw,\bv)}^2_{{\bX_{1,T}}} 
&\lesssim \bE \left[\left(  \int_0^T \norm{\bw}_{\dot{H}_2^1}^4 dt\right)^{1/2}  \left(  \int_0^T \norm{\bv}_{\dot{H}^1}^4 dt\right)^{1/2} \right] \\
&\lesssim \|\bv \|_{\bX_{2,T}}^2 \|\bw\|_{\bX_{2,T}}^2.
\end{align*}
\end{proof}

We now prove that the bilinear operator is jointly continuous on $\bL_5(t) \times \bL_5(t)$.
We divide its proof into a few steps.

\begin{proposition}
					\label{bi est 2}
There exists a constant $C_{11}>0$ such that for all $T>0$ and $\bv, \bw \in \bL_5(t)$,
\[
\|B(\bv,\bw)\|_{\bL_5(T)} \le C_{11} \|\bv\|_{\bL_5(T)}\|\bw\|_{\bL_5(T)}.
\]
\end{proposition}

\begin{proof}
\begin{enumerate}[\bf{Step} 1)]
\item
To prove this proposition, it is useful to have an integral representation 
\begin{align*}
B(\bv,\bw)(t)
&=\int_0^t S(t-s) \bP \nabla \cdot (\bv \otimes \bw)(s) ds \\
&=\int_0^t K(t-s) \ast (\bv \otimes \bw)(s) ds
\end{align*}
with a kernel $K$.

Let $h$ be a $L_{p}(\fR^3)$-valued function  defined on $(0,T)$ with $T>0$ and $p \in (1,\infty)$.
For each $s \in (0,T)$, $h(s) \in L_{p}(\fR^3)$.
That is, fixing $s$, we can regard $h(s)$  as a function defined on $\fR^3$ and
by  $h(s,y)$ we denote the value of this function at $y \in \fR^3$. 
The kernel $K$ should satisfy
\[
K(t-s,\cdot) * h(s, \cdot) = S(t-s) \bP \nabla h(s) = S(t-s) (I + \cR \otimes \cR) \nabla h (s)
\]
so that we can write its $j$-th component of $K$ as an inverse Fourier transform of the multiplier  
\[
\left(K(t,x) \right)^j =  -i\sum_{k,m \in \{1,2,3\}} \cF^{-1} \left[ e^{-t|\xi|^2} |\xi|^{-2} \xi_j \xi_k \xi_m \right](x).
\]
The key points of this kernel representation are the following scaling property
\begin{equation}
\label{E41}
K(t,x) = t^{-2} K(1,t^{-1/2}x)
\end{equation}
and the pointwise decay property 
\begin{equation}
\label{E42}
\sup_{x \in \fR^3} (1+|x|)^{4} |K(1,x)| < \infty.
\end{equation}
Ossen studied this kind of kernel and now its properties are regarded as well-known facts.
We refer the reader to \cite{Ler2009} for its generalization and other fine properties.
For reader's convenience we provide a proof of these two basic facts at Step 3.
\item
We use the Young inequality and the scaling property \eqref{E41} to obtain that 
\begin{align*}
&\norm{B(\bv,\bw)(t)}_{L^5(\fR^3)} \\
&\lesssim \int_0^t \norm{K(t-s,\cdot) * (\bv \otimes \bw)(s)}_{L_5(\fR^3)} ds \\
&\lesssim \int_0^t \norm{K(t-s,\cdot)}_{L_{5/4}(\fR^3)} \norm{(\bv \otimes \bw)(s)}_{L_{5/2}(\fR^3)} ds \\
&\lesssim \int_0^t (t-s)^{-4/5} \norm{K(1,\cdot)}_{L_{5/4}(\fR^3)} \norm{(\bv \otimes \bw)(s)}_{L_{5/2}(\fR^3)} ds.
\end{align*}
Since $\norm{K(1,\cdot)}_{L_{5/4}(\fR^3)} < \infty$ from the pointwise decay property \eqref{E42}, we use the Hardy--Littlewood--Sobolev theorem on fractional integration (cf. \cite[Theorem 6.1.3]{Grafa2009}) to obtain that 
\begin{align*}
&\int \norm{B(\bv,\bw)(t)}_{L^5(\fR^3)}^5 dt \\
&\lesssim \int_0^T \left(\int_0^t (t-s)^{-4/5} \norm{(\bv \otimes \bw)(s)}_{L_{5/2}(\fR^3)} ds\right)^5 dt \\
&\lesssim \left(\int_0^T \norm{(\bv \otimes \bw)(t)}_{L_{5/2}(\fR^3)}^{5/2} dt\right)^2.
\end{align*}
Since $\norm{\bv \otimes \bw}_{\bL_{5/2}(T)} \le \norm{\bv}_{\bL_5(T)} \norm{\bw}_{\bL_5(T)}$, we use the Cauchy--Schwarz inequality inequality to get the result.
\item
Finally, we prove \eqref{E41} and \eqref{E42}.
We denote
\[
K_{j,k,m}(t,x) = \cF^{-1}\left[e^{-t|\xi|^2} |\xi|^{-2} \xi_j \xi_k \xi_m\right].
\]
Since we have 
\[
\left(K(t,x) \right)^j = -i\sum_{k,m}K_{j,k,m}(t,x),
\]
it suffices to prove that $K_{j,k,m}(t,x)$ satisfies the same properties as \eqref{E41} and \eqref{E42}.
We notice that 
\begin{align*}
\partial_t K_{j,k,m}(t,x) 
&= - \int_{\fR^3} e^{ix \cdot \xi} e^{-t|\xi|^2} \xi_j \xi_k \xi_m d\xi \\
&= i D_x^j D_x^k D_m \int_{\fR^3} e^{ix \cdot \xi} e^{-t|\xi|^2} d\xi \\
&= i D_x^j D_x^k D_x^m (\pi/t)^{3/2} e^{-\pi^2 |x|^2/t}
\end{align*}
and $K_{j,k,m}(t,x)$ goes to $0$ as $t \to \infty$.
Thus, by the Fundamental theorem, of calculus, we have 
\begin{align*}
K_{j,k,m}(t,x) 
&= i \int_t^\infty D_x^j D_x^k D_x^m (\pi/s)^{3/2} e^{-\pi^2 |x|^2/s} ds \\
&= i \pi^{15/2} x_j x_k x_m \int_t^\infty s^{-9/2} e^{-\pi^2 |x|^2/s} ds.
\end{align*}
By the change of variables $\pi^2 |x|^2/s = a$, we get 
\begin{align*}
K_{j,k,m}(t,x) 
&= i \pi^{15/2} x_j x_k x_m \int_t^\infty s^{-7/2} e^{-\pi^2 |x|^2/s} \frac{ds}{s} \\
&= i \pi^{15/2} x_j x_k x_m |x|^{-7} \int_0^{\pi^2 |x|^2/t} a^{9/2} e^{-a} da.
\end{align*}
From this identity it is easy to check the scaling 
\[
K_{j,k,m}(t,x) = t^{-2} K_{j,k,m}(1,t^{-1/2}x),
\]
and the decay 
\[
|K_{j,k,m}(1,x)| \lesssim (1+|x|)^{-4}.
\]

\end{enumerate}
\end{proof}

\section{Proof of main theorems}

The proofs of Theorem \ref{main thm 3} and Theorem \ref{main thm 1} go along the same series of steps.
In order to write the procedures in a tidy manner, we shall prove the following abstract fixed point lemma.

We recall that if $\bX$ is a Banach space with the norm $\|\cdot \|_{\bX}$, then the product space $\bX \times \bX$ is a Banach space with the norm
\[
\| (\bv, \bw)\|_{\bX \times \bX} = \| \bv \|_{\bX} + \|\bw \|_{\bX}.
\]
We shall use the Banach fixed point theorem to the product space $\bX \times \bX$.

\begin{lemma}
				\label{cont lem}
Let $\bX$ be a Banach space with the norm $\| \cdot \|_\bX$.
Assume that $\bB(v,w)$ be a jointly continuous bilinear operator on $\bX  \times \bX$, that is, there exists a positive constant $C_\bX$ such that forall $\bv, \bw \in \bX$,
\begin{align}
					\label{e 0202 1}
\|\bB(\bv,\bw)\|_\bX \leq C_{\bX} \|\bv\|_\bX \|\bw\|_\bX.
\end{align}
If $\varepsilon < (4C_{\bX})^{-1}$,
\begin{equation}
					\label{e 0202 2}
\|\bv^1\|_{\bX} + \|\bw^1\|_\bX  \leq \varepsilon,
\end{equation}
and $(\bv^n, \bw^n)$ is a sequence in $\bX \times \bX$ satisfying 
\begin{equation}
					\label{e 0202 0}
\begin{split}
\bv^{n+1} & =  \bv^1-\bB(\bw^n,\bv^n) \\
\bw^{n+1} & =  \bw^1- \bB(\bv^n,\bw^n),
\end{split}
\end{equation}
then there exists a unique $(\bv,\bw) \in \bX  \times \bX$ satisfying 
\[
\|\bv\|_{\bX} + \|\bw\|_\bX  \leq 2\varepsilon.
\]
and 
\begin{equation}
					\label{e 0203 1}
\begin{split}
\bv & =  \bv^1-\bB(\bw,\bv) \\
\bw & =  \bw^1-\bB(\bv,\bw).
\end{split}
\end{equation}
\end{lemma}

\begin{proof}
We note first that for all $n \in \bN$, 
\begin{align}
					\label{e 0202 5}
\|\bv^{n}\|_\bX + \|\bw^{n}\|_\bX
\leq 2\varepsilon.
\end{align}
Indeed, mathematical induction yields 
\begin{align*}
&\|\bv^{n+1}\|_\bX + \|\bw^{n+1}\|_\bX \\
&\leq \|\bv^{1}\|_\bX + \|\bw^{1}\|_\bX   
+\|\bB(\bw^n,\bv^n)\|_\bX + \|\bB(\bv^n,\bw^n)\|_\bX \\
&\leq \varepsilon + 2 C_{\bX} \|\bv^n\|_\bX \|\bw^n\|_\bX \\
&\leq \varepsilon + 2 C_{\bX} \varepsilon^2 \\
&\leq 2\varepsilon.
\end{align*}

Since $\bB$ is bilinear, we may write for $n\ge2$,
\begin{align*}
\bv^{n+1} - \bv^n
& =  -\bB(\bw^{n},\bv^{n}) +\bB(\bw^{n-1},\bv^{n-1}) \\
& =-\bB(\bw^{n}, \bv^{n}-\bv^{n-1})- \bB(\bw^{n}-\bw^{n-1},\bv^{n-1})
\end{align*}
and 
\begin{align*}
\bw^{n+1} -\bw^n
& =  -\bB(\bv^{n},\bw^{n})+\bB(\bv^{n-1},\bw^{n-1})  \\
& =-\bB(\bv^{n}, \bw^{n}-\bw^{n-1})-\bB(\bv^{n}-\bv^{n-1},\bw^{n-1}) .
\end{align*}
Thus, we have 
\begin{align*}
&\|(\bv^{n+1}, \bw^{n+1}) - (\bv^n, \bw^n)\|_{\bX \times \bX} \\
&= \|\bv^{n+1} -  \bv^n\|_{\bX}
+ \|\bw^{n+1} - \bw^n\|_{\bX} \\
& \leq 4\varepsilon C_{\bX} (\|\bv^n - \bv^{n-1}\|_{\bX} + \|\bw^n  - \bw^{n-1}\|_{\bX}) \\
&< \delta \|(\bv^n, \bw^n) - (\bv^{n-1}, \bw^{n-1})\|_{\bX \times \bX}
\end{align*}
by \eqref{e 0202 5}, where $\delta$ is a constant satisfying $4 \varepsilon C_{\bX} < \delta < 1$.

Therefore, the Banach fixed-point theorem implies that $(\bv^n , \bw^n)$ converges to a unique $(\bv , \bw) \in \bX \times \bX$ satisfying \eqref{e 0203 1}, which follows by taking the limit to \eqref{e 0202 0} and \eqref{e 0202 5}.
\end{proof}

Now, we are ready to prove Theorem \ref{main thm 1}.

\begin{proof}[\bf{Proof of Theorem \ref{main thm 1}}]
Recall 
\begin{align*}
{\bX_{1,T}} &:=  L_2 \left( \Omega, \rF ; C([0,T]:\dot{H}_2^{1/2 }(\R)) \right) \cap \dot\bH^{3/2}_2(T), \\ 
{\bX_{2,T}} &:=  L_4 \left( \Omega, \rF ; L_4([0,T]:\dot{H}_2^{1}(\R)) \right).
\end{align*}
From \eqref{E24}, Corollary \ref{sto heat cor 1}, and Proposition \ref{bi est 1}, we have 
\begin{align*}
&\|(\bv,\bw)\|_{\bX_{1,T} \times \bX_{1,T}} \\
&\lesssim \|(\bv_0,\bw_0)\|_{\dot U_p^{1/2} \times \dot U_p^{1/2}} 
+ \|(G_1,G_2)\|_{\dot \bH_2^{1/2}(T,l_2) \times \dot \bH_2^{1/2}(T,l_2)} \\
&\quad + \| \bv\|_{\bX_{2,T}}\| \bw\|_{\bX_{2,T}}.
\end{align*}
Thus, it suffices to find a  solution $(\bv,\bw) \in \bX_{2,T} \times \bX_{2,T}$.

We set 
\begin{align*}
&\bv^1(t,x) = S(t)\bv_0(x) + \int_0^t S(t-s)G_1(s,\cdot)(x)dw^k_s\\
&\bw^1(t,x) =  S(t)\bw_0(x) +\int_0^t S(t-s)G_2(s,\cdot)(x)dw^k_s 
\end{align*}
and inductively define $(\bv^n, \bw^n)$ for $n \ge 2$ by 
\begin{align*}
\bv^{n}(t,x) 
= S(t)\bw_0(x) - B(\bw^{n-1},\bv^{n-1})
+ \int_0^t S(t-s)G_1(s,\cdot)dw^k_s \\
\bw^{n}(t,x) 
= S(t)\bv_0(x) - B(\bv^{n-1},\bw^{n-1})
 + \int_0^t S(t-s) G_2(s,\cdot)dw^k_s.
\end{align*}

By mathematical induction, we have for all $n \in \bN$, 
\begin{align}
					\label{2018 0323 1}
\bv^{n}, \bw^{n} \in \bX_{2,T}.
\end{align}
Indeed, due to \eqref{E31} and Corollary \ref{sto heat cor 1}, we have $\bv^{1}, \bw^1 \in \bX_{2,T}$.
If we assume that \eqref{2018 0323 1} holds with $n=k$, then by  \eqref{E31}, Corollary \ref{sto heat cor 1}, and Proposition \ref{bi est 1},
\begin{align*}
&\|(\bv^{k+1},\bw^{k+1})\|_{\bX_{2,T} \times \bX_{2,T}} \\
&\lesssim \|(\bv_0,\bw_0)\|_{\dot U_p^{1/2} \times \dot U_p^{1/2}} 
+ \|(G_1,G_2)\|_{\dot \bH_2^{1/2}(T,l_2) \times \dot \bH_2^{1/2}(T,l_2)} 
+ \|\bv^k\|_{\bX_{2,T}}\| \bw^k\|_{\bX_{2,T}} \\
&< \infty.
\end{align*}
This proves \eqref{2018 0323 1}.

Finally, we can apply Lemma \ref{cont lem} due to \eqref{E31}, Corollary \ref{sto heat cor 1} and Proposition \ref{bi est 1} so that we get the second part, Theorem \ref{main thm 1}(ii).

To obtain the first part, Theorem \ref{main thm 1}(i), it suffices to show that 
there exists a positive constant $T_0 \in (0,\infty)$ such that
\begin{align*}
\|v^1\|_{\bX_{2,T_0}} + \|w^1\|_{\bX_{2,T_0}}  < \frac{1}{ 4 C_{10}},
\end{align*}
which is an immediate consequence of the fact that 
\[
\lim_{T \to 0} \|v^1\|_{\bX_{2,T}} + \|w^1\|_{\bX_{2,T}} = 0.
\]
This completes the proof of Theorem \ref{main thm 1}.
\end{proof}

The proof of Theorem \ref{main thm 3} is almost identical to the proof of Theorem \ref{main thm 1} except that Corollary \ref{sto heat cor 2} and Proposition \ref{bi est 2} are used in place of Corollary \ref{sto heat cor 1} and Proposition \ref{bi est 1}.

\section*{Acknowledgements}

I. Kim's work was supported by the National Research Foundation of Korea(NRF) grant  funded by the Korea government(MSIT) (No. 2017R1C1B1002830).
M. Yang's work was partially supported by the National Research Foundation of Korea(NRF) grant  funded by the Korea government(MSIT) (No. 2015R1A5A1009350) and by the Yonsei University Research Fund (No. 2018-22-0046).



\begin{thebibliography}{10}

\bibitem{Davidson}
P.~A. Davidson.
\newblock {\em An Introduction to Magnetohydrodynamics, Cambridge Texts Appl. Math., Cambridge University Press}, Cambridge, 2001.


\bibitem{fujita1964navier}
H.~Fujita and T.~Kato.
\newblock On the navier-stokes initial value problem. i.
\newblock {\em Archive for rational mechanics and analysis}, 16(4):269--315,
  1964.


\bibitem{Grafa2009}
L.~Grafakos.
\newblock Modern Fourier Analysis.
\newblock {\em Springer}, second edition, 2009.


\bibitem{JLinter1976}
B. Jöran, and J.Löfström.
\newblock  Interpolation spaces: an introduction. 
\newblock {\em Springer-Verlag Berlin Heidelberg}, 1976.


\bibitem{Krylov1999}
N.~V. Krylov.
\newblock An analytic approach to {SPDEs}.
\newblock {\em Stochastic Partial Differential Equations: Six Perspectives,
  Mathematical Surveys and Monographs}, 64:185--242, 1999.


\bibitem{van2007stochastic}
J.~M. van Neerven, M.~C. Veraar, L.~Weis.
\newblock Stochastic integration in UMD Banach spaces.
\newblock {\em The Annals of Probability}, 35(4):1438--1478, 2007.


\bibitem{van2012stochastic}
J.~M. van Neerven, M.~C. Veraar, L.~Weis.
\newblock Stochastic maximal $L_p$-regularity
\newblock {\em The Annals of Probability}, 40(2): 788-812, 2012.


\bibitem{Ler2009}
N. Lerner,
\newblock A note on the Oseen kernels.
\newblock {\em Advances in Phase Space Analysis of Partial Differential Equations. }, Birkhäuser Boston, 2009. p. 161-170.


\bibitem{Tan2016}
Z. Tan, D. Wang, and H. Wang,
\newblock  Global strong solution to the three-dimensional stochastic incompressible magnetohydrodynamic equations.
\newblock {\em Mathematische Annalen},  365(3-4), 1219-1256, 2016.


\end{thebibliography}
\end{document}